\documentclass[12pt]{amsart}
\usepackage{amssymb}
\usepackage{amsfonts}
\usepackage{latexsym}
\usepackage{amscd}
\usepackage[mathscr]{euscript}
\usepackage{xy} \xyoption{all}
\usepackage[utf8]{inputenc} 

\usepackage{tikz}
\usetikzlibrary{shapes,arrows}
\usepackage[utf8]{inputenc}
\usepackage{color}
\newcount\parracount
\newcount\recount
\newtheorem{definition}{Definition}
\newtheorem{theorem}[definition]{Theorem}
\newtheorem{lemma}[definition]{Lemma}
\newtheorem{corollary}[definition]{Corollary}
\newtheorem{remark}[definition]{Remark}
\newtheorem{fact}[definition]{Fact}
\newtheorem{example}[definition]{Example}

\newtheorem{proposition}[definition]{Proposition}

\newenvironment{2ndproof}{\paragraph{{\it Second Proof. }}}{\hfill$\square$\\}

\usepackage[utf8]{inputenc}
\usepackage{color}
\usepackage{setspace}

\def\Q{\mathbb Q}
\def\N{\mathbb N}
\def\Z{\mathbb Z}

\begin{document}

\title[IBN]{Cohn-Leavitt Path Algebras and the Invariant Basis Number Property}

\author{Müge Kanuni}
 \address{Department of Mathematics, Düzce University, 
Konuralp Düzce 81620 Turkey} \email{muge.kanuni@duzce.edu.tr}
\author{Murad Özaydın}
\address{Department of Mathematics, University of Oklahoma,
Norman OK  U.S.A.} \email{mozaydin@ou.edu}


%\begin{abstract}
%We give the necessary and sufficient condition for a separated Cohn-Leavitt path algebra of a finite digraph to have IBN. As a consequence, %separated Cohn path algebras have IBN. We determine the non-stable K-theory of a corner ring in terms of the non-stable 
%K-theory of the ambient ring. We give a necessary condition for a corner algebra of a separated Cohn-Leavitt path algebra of a finite graph to have IBN. We provide Morita equivalent rings which are non-IBN, but are of different types.
%\end{abstract}

\maketitle

\date{}

\maketitle
\onehalfspacing
\section{Introduction}

Leavitt Path Algebras were defined just over a decade ago \cite{AA1,AMP}, but they have roots in the works of Leavitt in the 60s focused on understanding the nature of the failure of the IBN (Invariant Basis Number) property for arbitrary rings \cite{L1}.

\medskip

A ring has IBN if any two bases of a finitely generated free module have the same number of elements. Fields, division rings, commutative rings, Noetherian rings all have IBN. A classical example of a ring without IBN is the algebra of endomorphisms of a countably infinite dimensional vector space. The free module of rank 1 over this ring has bases of $n$ elements for any positive integer $n$. In the early 60s William Leavitt asked and then answered this question: Given  any positive integers $m<n$ is there a ring $R$ having a free module with a basis of $m$ elements and another basis with $n$ elements but no bases with $k$ elements if $k<n$ and not equal to $m$? Such a ring is said to be non-IBN of type $(m,n)$. 

\medskip

The algebras Leavitt constructed are denoted by $L(m,n)$ and nowadays called the Leavitt algebras. They have a semi-universal property: for any algebra $A$ satisfying $A^m \cong A^n$ as $A$-modules there is an algebra homomorphism (not necessarily
 unique) from $L(m,n)$ to $A$. They are simple if and only if $m=1$, \cite{L2}. The fact that a free $L(m,n)$ module of rank m has a basis of $n$ elements is immediate from Leavitt's definition. However, to show that there are no bases of $k$ elements with $m<k<n$ requires an algebraic invariant, namely $\mathcal{V}(R)$ usually referred to as nonstable $\mathcal{K}$-theory. 

\medskip

$\mathcal{V}(R)$ is the monoid of isomorphism classes of finitely generated projective $R$-modules under direct sum. The Grothendieck group of 
$\mathcal{V}(R)$ is denoted by $\mathcal{K}_0(R)$. It turns out that $\mathcal{K}_0(R)$ (in fact, $\mathcal{K}_0(R) \otimes \Q$) suffices to determine whether $R$ has IBN or not ( \cite[Corollary 5.7]{AGP}, also Proposition \ref{RIBNU(R)inf} below). However, $\mathcal{K}_0(R)$ cannot detect the type $(m,n)$ when $R$ is non-IBN. Another reason to work with $\mathcal{V}(R)$ rather than $\mathcal{K}_0(R)$ is to understand the 
$\mathcal{K}$-theory of corner subrings of $R$. 

\medskip

For any idempotent $\alpha \in R$, $\alpha R \alpha$ is the corner ring associated with $\alpha$. While $\mathcal{V}(\alpha R \alpha)$ is 
(isomorphic to) a submonoid of $\mathcal{V}(R)$, the corresponding does not hold for the $\mathcal{K}_0$ groups (Theorem \ref{V(xRx)1-1V(R)} and Example \ref{Toeplitz2} below). 
In particular, $\mathcal{K}_0(R)$ cannot detect whether $\alpha R \alpha$ has IBN or not, but $\mathcal{V}(R)$ does, as well as is determining its type $(m,n)$ when it is non-IBN.

\medskip 
In general, it is difficult to compute $\mathcal{K}_0(R)$, even harder to compute $\mathcal{V}(R)$ for an arbitrary ring. However, in \cite{B} 
George Bergman gave a construction for a ring $R$ whose $\mathcal{V}(R)$ is any given commutative finitely generated monoid satisfying two obviously necessary conditions. Leavitt path algebras were shown to fit the Bergman scheme in \cite{AMP} and this was extended to Cohn-Leavitt 
path algebras of a separated di(rected )graph in \cite{AG}. Specifically, they showed that $\mathcal{V}(R)$ is isomorphic to the monoid of the (separated) digraph which is given explicitly with generators and relations \cite[Theorem 4.3]{AG}. 

\medskip 

The main purpose of this note is to give an easily checked criterion to determine whether a separated Cohn-Leavitt path algebras has IBN (Theorem \ref {MainThmforseparatedLPA}). An immediate consequence is that separated Cohn path algebras have IBN, which generalizes the main result in \cite{AK}. We also describe the non-stable $\mathcal{K}$-theory of a corner ring explicitly from the non-stable $\mathcal{K}$-theory of the ring (Theorem \ref{V(xRx)1-1V(R)}).

\medskip

A proof of Theorem \ref {MainThmforseparatedLPA} in the special case of Leavitt path algebras was presented during the CIMPA 2015 
Research School in Izmir by the first named author. The videos of all the lectures are publicly available at the website 
http://nesinkoyleri.org/eng/events/2015-lpa/. Below we give a more algebraic proof, but we also sketch the original proof.

\medskip

In Section 2 below we give a geometric representation of a commutative semigroup defined by generators and relations. We also define graph semigroups, the Grothendieck group construction and a closure operator on a semigroup. We list some basic properties that we will use and give a few examples and counterexamples.
Proposition \ref{G(U)embedG(V)} proves that the Grothendieck group of the semigroup of isomorphism classes of finitely generated free 
$L$-modules over a separated Cohn-Leavitt path algebra $L$ embeds in $\mathcal{K}_0(L)$. 

\medskip

Section 3 and 4 are the core of our paper. In Section 3, Theorem \ref{MainThmforseparatedLPA} gives the necessary and sufficient condition for a separated Cohn-Leavitt path algebra of a finite digraph to have IBN. As a consequence, we show that the separated Cohn path algebras have IBN (Corollary \ref{sepCohnIBN}). 

\medskip

In Section 4, Theorem \ref{V(xRx)1-1V(R)} gives the non-stable 
$\mathcal{K}$-theory of a corner ring explicitly in terms of the non-stable $\mathcal{K}$-theory of the ring. It proves that 
$\mathcal{V}(\alpha R \alpha)$ embeds in $\mathcal{V}(R)$ and also shows that $\mathcal{V}(\alpha R \alpha)$ is isomorphic to the closure of the semigroup generated by the isomorphism class of $\alpha R$ in $\mathcal{V}(R)$. Theorem \ref{MainThmforCornerLPA} states a necessary condition for a corner algebra of a separated Cohn-Leavitt path algebra of a finite graph to have IBN. We also provide an example to show that this condition is not sufficient in general. In Example \ref{E1}, we provide Morita equivalent rings which are non-IBN, but are of different types.
      
%Section 5 is devoted to examples. In Example \ref{E1}, we provide Morita equivalent rings which are non-IBN, but are of different types. We also provide several examples of Leavitt path algebras (both IBN and non-IBN) satisfying 
%the same ring theoretical conditions which are characterized by graph theoretical properties. 
% On the other hand, there are graph theoretical properties (such as : having finite GK-dimension equivalently all cycles are disjoint 
%\cite{AAJZ}) whose LPA are all IBN. 
%These seem to indicate that a graph theoretical condition equivalent to a Leavitt Path Algebras having IBN will necessarily be rather complicated. 
 
\section{Preliminaries}
\subsection{Finitely Presented Commutative Semigroups}

Let $\Omega$ be a finite set of generators and $$\{ x_i = y_i : \, x_i, y_i \in \N\Omega \setminus \{{ \mathbf 0}\} \}_{i=1}^n$$ be the relations. Consider the graph with vertex set $\N\Omega  \setminus \{{ \mathbf 0} \}$ . 
%define the free commutative semigroup 
%and impose finitely many relations $$x_i = y_i \mbox{ where } x_i, y_i \in \N\Omega \setminus \{{ \mathbf 0}, \mbox{ for } i=1,...,n.$$  
Each relation $x_i = y_i$, gives an edge $W_i$ joining $x_i$ to $y_i$. Let the edges of this graph be all translates of the $W_i$, $i=1,...,n$ by elements of $\N\Omega  \setminus \{{ \mathbf 0} \}$.  The path components of this graph yield the geometric representation of the semigroup $S$ generated by $\Omega$ subject to these relations. (If the nodes are connected via any sequence of edges, then they are in the same equivalence class. Each path component corresponds to a distinct element of $S$.) The addition on $S$ descends from the addition of $\N\Omega  \setminus \{{ \mathbf 0} \}$.
\begin{example} \label{geoL(2,5)} 
Take $\Omega= \{v,w \}$, and the relations $R_1 : \, \, v = 2w$ and $R_2 : \, \, v =5w$. To emphasize the difference we mark the relation and its translates of $R_1$ with blue and $R_2$ with red. Then, the geometric representation of 
$\N\Omega  \setminus \{{ \mathbf 0} \}$ subject to $\{ R_1,R_2\}$ is given by 
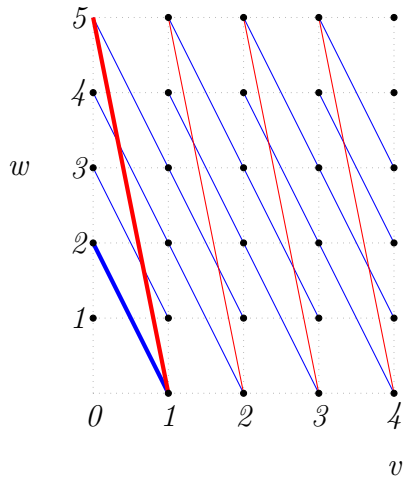
\begin{figure}[h]
\begin{tikzpicture}
\draw[dotted, help lines] (0,0) grid (4,5);
\draw[-, blue, ultra thick] (1,0) -- (0,2);
\draw[-, blue] (1,1) -- (0,3);
\draw[-, blue] (1,2) -- (0,4);
\draw[-, blue] (2,0) -- (1,2);
\draw[-, blue] (2,1) -- (1,3);
\draw[-, blue] (2,2) -- (1,4);
\draw[-, blue] (3,0) -- (2,2);
\draw[-, blue] (3,1) -- (2,3);
\draw[-, blue] (3,2) -- (2,4);
\draw[-, blue] (3,3) -- (2,5);
\draw[-, blue] (2,3) -- (1,5);
\draw[-, blue] (1,3) -- (0,5);
\draw[-, blue] (4,0) -- (3,2);
\draw[-, blue] (4,1) -- (3,3);
\draw[-, blue] (4,2) -- (3,4);
\draw[-, blue] (4,3) -- (3,5);
\draw[-, red, ultra thick] (1,0) -- (0,5);
%\draw[-, red] (1,1) -- (0,6);
%\draw[-, red] (1,2) -- (0,7);
\draw[-, red] (2,0) -- (1,5);
%\draw[-, red] (2,1) -- (1,6);
%\draw[-, red] (2,2) -- (1,7);
\draw[-, red] (3,0) -- (2,5);
%\draw[-, red] (3,1) -- (2,6);
%\draw[-, red] (3,2) -- (2,7);
%\draw[-, red] (3,3) -- (2,8);
%\draw[-, red] (2,3) -- (1,8);
%\draw[-, red] (1,3) -- (0,8);
\draw[-, red] (4,0) -- (3,5);
%\draw[-, red] (4,1) -- (3,6);
%\draw[-, red] (4,2) -- (3,7);
%\draw[-, red] (4,3) -- (3,8);
\node at (1,-0.3) {1};
\node at (2,-0.3) {2};
\node at (3,-0.3) {3};
\node at (4,-0.3) {4};
\node at (0,-0.3) {0};
\node at (-0.2,1) {1};
\node at (-0.2,2) {2};
\node at (-0.2,3) {3};
\node at (-0.2,4) {4};
\node at (-0.2,5) {5};
\node at (4,-1) {v};
\node at (-1,3) {w};
%\node at (2,-2) {$S(\Gamma)$};
\draw [fill=black] (1,0) circle [radius=0.04];
\draw [fill=black] (1,1) circle [radius=0.04];
\draw [fill=black] (1,2) circle [radius=0.04];
\draw [fill=black] (1,3) circle [radius=0.04];
\draw [fill=black] (1,4) circle [radius=0.04];
\draw [fill=black] (1,5) circle [radius=0.04];
\draw [fill=black] (2,0) circle [radius=0.04];
\draw [fill=black] (2,1) circle [radius=0.04];
\draw [fill=black] (2,2) circle [radius=0.04];
\draw [fill=black] (2,3) circle [radius=0.04];
\draw [fill=black] (2,4) circle [radius=0.04];
\draw [fill=black] (2,5) circle [radius=0.04];
\draw [fill=black] (3,0) circle [radius=0.04];
\draw [fill=black] (3,1) circle [radius=0.04];
\draw [fill=black] (3,2) circle [radius=0.04];
\draw [fill=black] (3,3) circle [radius=0.04];
\draw [fill=black] (3,4) circle [radius=0.04];
\draw [fill=black] (3,5) circle [radius=0.04];
\draw [fill=black] (4,0) circle [radius=0.04];
\draw [fill=black] (4,1) circle [radius=0.04];
\draw [fill=black] (4,2) circle [radius=0.04];
\draw [fill=black] (4,3) circle [radius=0.04];
\draw [fill=black] (4,4) circle [radius=0.04];
\draw [fill=black] (4,5) circle [radius=0.04];
%\draw [fill=black] (0,0) circle [radius=0.04];
\draw [fill=black] (0,1) circle [radius=0.04];
\draw [fill=black] (0,2) circle [radius=0.04];
\draw [fill=black] (0,3) circle [radius=0.04];
\draw [fill=black] (0,4) circle [radius=0.04];
\end{tikzpicture}
\end{figure}

\end{example}
If we take $\Z\Omega$ as vertices and translates of the $W_i$ above by $\Z\Omega$ as edges, then the connected components give the Grothendieck group of the semigroup $S$ to be explained below. 
\begin{example} \label{geoToeplitz}
Take $\Omega= \{v,w \}$, and the relation $R_1 : \, \, v = v +w$. The geometric representation of 
$\N\Omega  \setminus \{{ \mathbf 0} \}$ subject to $\{ R_1\}$ is given by the figure on the left and $\Z \Omega / < R_1 >$ is given by the figure on the right.
\begin{figure}[h]
\begin{tikzpicture}
\draw[dotted, help lines] (0,0) grid (4,4);
\draw[-, blue, ultra thick] (1,0) -- (1,1);
\draw[-, blue] (1,1) -- (1,2);
\draw[-, blue] (1,2) -- (1,3);
\draw[-, blue] (2,0) -- (2,1);
\draw[-, blue] (2,1) -- (2,2);
\draw[-, blue] (2,2) -- (2,3);
\draw[-, blue] (3,0) -- (3,1);
\draw[-, blue] (3,1) -- (3,2);
\draw[-, blue] (3,2) -- (3,3);
\draw[-, blue] (3,3) -- (3,4);
\draw[-, blue] (2,3) -- (2,4);
\draw[-, blue] (1,3) -- (1,4);
\draw[-, blue] (4,0) -- (4,1);
\draw[-, blue] (4,1) -- (4,2);
\draw[-, blue] (4,2) -- (4,3);
\draw[-, blue] (4,3) -- (4,4);
\node at (1,-0.3) {1};
\node at (2,-0.3) {2};
\node at (3,-0.3) {3};
\node at (4,-0.3) {4};
\node at (0,-0.3) {0};
\node at (-0.2,1) {1};
\node at (-0.2,2) {2};
\node at (-0.2,3) {3};
\node at (-0.2,4) {4};
\node at (4,-1) {v};
\node at (-1,3) {w};
%\node at (2,-2) {$S(\Gamma)$};
\draw [fill=black] (1,0) circle [radius=0.04];
\draw [fill=black] (1,1) circle [radius=0.04];
\draw [fill=black] (1,2) circle [radius=0.04];
\draw [fill=black] (1,3) circle [radius=0.04];
\draw [fill=black] (1,4) circle [radius=0.04];
\draw [fill=black] (2,0) circle [radius=0.04];
\draw [fill=black] (2,1) circle [radius=0.04];
\draw [fill=black] (2,2) circle [radius=0.04];
\draw [fill=black] (2,3) circle [radius=0.04];
\draw [fill=black] (2,4) circle [radius=0.04];
\draw [fill=black] (3,0) circle [radius=0.04];
\draw [fill=black] (3,1) circle [radius=0.04];
\draw [fill=black] (3,2) circle [radius=0.04];
\draw [fill=black] (3,3) circle [radius=0.04];
\draw [fill=black] (3,4) circle [radius=0.04];
\draw [fill=black] (4,0) circle [radius=0.04];
\draw [fill=black] (4,1) circle [radius=0.04];
\draw [fill=black] (4,2) circle [radius=0.04];
\draw [fill=black] (4,3) circle [radius=0.04];
\draw [fill=black] (4,4) circle [radius=0.04];
%\draw [fill=black] (0,0) circle [radius=0.04];
\draw [fill=black] (0,1) circle [radius=0.04];
\draw [fill=black] (0,2) circle [radius=0.04];
\draw [fill=black] (0,3) circle [radius=0.04];
\draw [fill=black] (0,4) circle [radius=0.04];
\draw[dotted, help lines] (8,-2) grid (12,4);
\draw[-, blue, ultra thick] (10,0) -- (10,1);
\draw[-, blue] (10,1) -- (10,2);
\draw[-, blue] (10,2) -- (10,3);
\draw[-, blue] (10,3) -- (10,4);
\draw[-, blue] (11,0) -- (11,1);
\draw[-, blue] (11,1) -- (11,2);
\draw[-, blue] (11,2) -- (11,3);
\draw[-, blue] (11,3) -- (11,4);
\draw[-, blue] (12,0) -- (12,1);
\draw[-, blue] (12,1) -- (12,2);
\draw[-, blue] (12,2) -- (12,3);
\draw[-, blue] (12,3) -- (12,4);
\draw[-, blue] (9,0) -- (9,1);
\draw[-, blue] (9,1) -- (9,2);
\draw[-, blue] (9,2) -- (9,3);
\draw[-, blue] (9,3) -- (9,4);
\draw[-, blue] (9,-1) -- (9,0);
\draw[-, blue] (9,-2) -- (9,-1);
\draw[-, blue] (8,0) -- (8,1);
\draw[-, blue] (8,1) -- (8,2);
\draw[-, blue] (8,2) -- (8,3);
\draw[-, blue] (8,3) -- (8,4);
\draw[-, blue] (8,-1) -- (8,0);
\draw[-, blue] (8,-2) -- (8,-1);
\draw[-, blue] (9,-1) -- (9,0);
\draw[-, blue] (9,-2) -- (9,-1);
\draw[-, blue] (10,-1) -- (10,0);
\draw[-, blue] (10,-2) -- (10,-1);
\draw[-, blue] (11,-1) -- (11,0);
\draw[-, blue] (11,-2) -- (11,-1);
\draw[-, blue] (12,-1) -- (12,0);
\draw[-, blue] (12,-2) -- (12,-1);
\node at (9.8,-0.3) {1};
\node at (10.8,-0.3) {2};
\node at (11.8,-0.3) {3};
\node at (8.8,-0.3) {0};
\node at (8,-0.3) {-1};
\node at (8.8,-1) {-1};
\node at (8.8,-2) {-2};
\node at (8.8,1) {1};
\node at (8.8,2) {2};
\node at (8.8,3) {3};
\node at (8.8,4) {4};
\node at (13,0) {v};
\node at (9,4.5) {w};
%\node at (10,-3) {$S_{\Z}(\Gamma)$};
\draw [fill=black] (8,0) circle [radius=0.04];
\draw [fill=black] (8,1) circle [radius=0.04];
\draw [fill=black] (8,2) circle [radius=0.04];
\draw [fill=black] (8,3) circle [radius=0.04];
\draw [fill=black] (8,4) circle [radius=0.04];
\draw [fill=black] (9,0) circle [radius=0.04];
\draw [fill=black] (9,1) circle [radius=0.04];
\draw [fill=black] (9,2) circle [radius=0.04];
\draw [fill=black] (9,3) circle [radius=0.04];
\draw [fill=black] (9,4) circle [radius=0.04];
\draw [fill=black] (10,0) circle [radius=0.04];
\draw [fill=black] (10,1) circle [radius=0.04];
\draw [fill=black] (10,2) circle [radius=0.04];
\draw [fill=black] (10,3) circle [radius=0.04];
\draw [fill=black] (10,4) circle [radius=0.04];
\draw [fill=black] (11,0) circle [radius=0.04];
\draw [fill=black] (11,1) circle [radius=0.04];
\draw [fill=black] (11,2) circle [radius=0.04];
\draw [fill=black] (11,3) circle [radius=0.04];
\draw [fill=black] (11,4) circle [radius=0.04];
\draw [fill=black] (12,0) circle [radius=0.04];
\draw [fill=black] (12,1) circle [radius=0.04];
\draw [fill=black] (12,2) circle [radius=0.04];
\draw [fill=black] (12,3) circle [radius=0.04];
\draw [fill=black] (12,4) circle [radius=0.04];
\draw [fill=black] (12,-1) circle [radius=0.04];
\draw [fill=black] (12,-2) circle [radius=0.04];
%\draw [fill=black] (12,-3) circle [radius=0.04];
%\draw [fill=black] (12,-4) circle [radius=0.04];
\draw [fill=black] (11,-1) circle [radius=0.04];
\draw [fill=black] (11,-2) circle [radius=0.04];
%\draw [fill=black] (11,-3) circle [radius=0.04];
%\draw [fill=black] (11,-4) circle [radius=0.04];
\draw [fill=black] (10,-1) circle [radius=0.04];
\draw [fill=black] (10,-2) circle [radius=0.04];
%\draw [fill=black] (10,-3) circle [radius=0.04];
%\draw [fill=black] (10,-4) circle [radius=0.04];
\draw [fill=black] (9,-1) circle [radius=0.04];
\draw [fill=black] (9,-2) circle [radius=0.04];
%\draw [fill=black] (9,-3) circle [radius=0.04];
%\draw [fill=black] (9,-4) circle [radius=0.04];
\draw [fill=black] (8,-1) circle [radius=0.04];
\draw [fill=black] (8,-2) circle [radius=0.04];
\end{tikzpicture}
%\caption{$S(\Gamma)$ \hspace{6 cm} $S_{\Z}(\Gamma)$}
\end{figure}

\end{example}

\subsection{Grothendieck Groups} 
\subsubsection{Closure operator on a semigroup}

An element $s$ of a semigroup $S$ is said to have {\bf infinite type} (or infinite order) if $s,s^2,s^3, \dots$ are distinct. Otherwise, $s$ is said to 
be {\bf torsion of type $(m,n)$} if $n$ is the least positive integer such that $x^n =x^m$ for some $0<m<n$.

If a semigroup generated by a single element is called {\bf cyclic} and it is determined (up to isomorphism) by the type of its generator. 

There is a pre-order $\preccurlyeq$ on a commutative (additive) semigroup $S$: $x \preccurlyeq y$ if and only if there is a 
$z \in S$ with $x + z =y$. 
For any subset $A$ of $S$, its {\bf closure} is  
$$\overline{A}:=\{x \in S : \, x \preccurlyeq s \mbox{ for some } s \in A\}. $$ 
This defines a closure operator on (the power set of) $S$, since: \\
(1) $A \subseteq \overline{A}$, \\
(2) $\overline{\overline{A}} = \overline{A}$, \\
(3) $ \overline{A \cup B} = \overline{A} \cup \overline{B}$. 

Note that if $A$ is a subsemigroup of $S$, then $\overline{A}$ is also a subsemigroup.
A set $F$ is { \bf closed},  if and only if  $a \preccurlyeq b \in F$ implies $a \in F$ for all $a,b \in S$. 

For any ring $R$, we denote by $\mathcal{V}(R)$ the additive commutative monoid of isomorphism classes of finitely generated projective (right) $R$-modules under direct sum. 
We denote the semigroup $\mathcal{V}(R) \setminus \{ [0]\}$ by $\mathcal{V}^*(R)$. Let $\mathcal{U}(R)$ be the cyclic submonoid of $\mathcal{V}(R)$, generated by $[R]$, hence $\mathcal{U}^*(R)$ is the semigroup of isomorphism classes of finitely generated nonzero free $R$-modules.
Note that $\mathcal{V}(R) = \overline{\mathcal{U}(R)}$. More generally, a finitely generated projective $R$-module $P$ is a {\it progenerator} if and only if $ \overline{<[P]>}= \mathcal{V}(R)$.

\subsubsection{The Grothendieck group}
The following discussion is basic and well-known \cite{R,Gr}, we included to fix notation and for completeness (until the referee tells us to cut it.)  

\begin{theorem}
\cite[Theorem 1.1.3]{R} Let $S$ be a multiplicative semigroup. There is an abelian (additive) group $G(S)$ (called the 
{\bf Grothendieck group} of $S$) together with a semigroup homomorphism $\iota_S: S \rightarrow G(S)$, such that for any (multiplicative) abelian group $H$, and any semigroup homomorphism $f: S \rightarrow H$, there exists a unique group homomorphism $\hat{f}: G(S) \rightarrow H$ with $f=\hat{f} \circ \iota_S$. 
\end{theorem}
\begin{proof}
The Grothendieck group of $S$, $G(S)$ is defined as $\mathbb{Z}S / K_S$ 
where $\mathbb{Z}S $ is a free abelian group with basis $S$, $K_S$ is a subgroup generated by $\{a+b-ab : a, b \in S \}$. The semigroup homomorphism $\iota_S : S \rightarrow G(S)$ maps $a \mapsto [a]=1a+K_S$ . Note that $\iota_S(ab)=[ab] = [a+b]= [a]+[b]=\iota_S(a)+\iota_S(b)$

Given any abelian group $H$ and any semigroup homomorphism $f:  S \longrightarrow H$, the map $\hat{f}:  G(S) \longrightarrow H$ is defined 
as $\displaystyle [\sum n_a a] \mapsto \prod_{a \in S} f(a)^{n_a}$. $\hat{f}$ is a homomorphism since 
$$\hat{f}(a+b-ab)=f(a)f(b)f(ab)^{-1}=1.$$ 

Since $\hat{f}([a])=f \iota_S(a) =f(a)$, the homomorphism $\hat{f}: G(S) \to H$ is determined on the generators of $G(S)$. Hence $\hat{f}$ is unique.
\end{proof}

We list a few relevant %well-known
properties of the Grothendieck group of a semigroup.  

\begin{fact}\label{PropertiesG(S)} Assume $S$, $T$ are multiplicative semigroups.
\begin{enumerate}
\item\label{GFunctor} $G$ is functorial :
 If  $f:S \rightarrow T$ is a semigroup homomorphism, 
then there is a group homomorphism $G(f)$ induced by $f$ between the Grothendieck groups 
$G(S)$ and $G(T)$.  
\item $G$ is left adjoint to the forgetful functor: For any abelian group $H$, we have a 1-1 correspondence %between $Hom(G(S),H)$ group homomorphisms, and $Hom(S,F(H))$ semigroup homomorphisms. 
$Hom(G(S),H) \leftrightarrow Hom(S,F(H))$ given by $f  \mapsto \hat{f}$ and
$g  \mapsto  g \circ \iota_S$ where $F(H)$ is $H$ regarded as a semigroup. 
%$ \xymatrix{
% Hom(G(S),H) \ar@{->}[r]   & Hom(S,H)  \ar@{->}[l]   \\
% f   &  \ar@{|->}[l] \hat{f} \\
%g  \ar@{|->}[r] & g \circ \iota_S}
%$
\item If $f: S \rightarrow T$ is onto, then $G(f): G(S) \rightarrow G(T)$ is also onto.
\item\label{ifScyclic,G(S)cyclic} If $S$ is a semigroup generated by $n$ elements; then $G(S)$ is finitely generated as a group with at most
$n$ generators. In particular, if $S$ is a cyclic semigroup, then $G(S)$ is a cyclic group.
\item\label{G(S)inf-iffSinf} Assume $S$ is a cyclic semigroup, then $S$ is a finite semigroup if and only if $G(S)$ is a finite group. 
\item\label{torsioninStorsioninG(S)} If an element $x \in S$ is torsion, then $\iota_S(x)$ is also torsion in $G(S)$, 
but the converse is not true.
\end{enumerate}
\end{fact}     
 Now for a given semigroup $S$, construct $S^+=S \cup \{ *\}$ as the monoid where multiplication on $S$ extends to $S^+$ with   
 $* \cdot s = s = s \cdot *$ for all $s \in S^+$.         
\begin{lemma}\label{SfgG(S)} Assume $S$ is a commutative multiplicative semigroup and $S^+$ is the monoid defined above. 
\begin{enumerate}
\item If $S$ is an abelian group, $\iota_S: S \to G(S)$ is an isomorphism. 
\item For any semigroup $S$, we have $G(S) \cong G(S^+)$. 
%\item If $S$ is a commutative group, $G(S) \cong G(S^+) \cong S$.
\item Suppose $M$ is a commutative multiplicative monoid for which $S= M - \{1\}$ is a group. Then $G(M) \cong S$. 
\end{enumerate}
\end{lemma}
              \begin{proof}
              \begin{enumerate}
\item Look at $id: S \to S$. Now by the definition of $G(S)$, the following diagram commutes
$$ \xymatrix{ 
 S \ar@{->}_{\iota_S}[d] \ar@{->}^{id}[r]  & S   &  &   a \ar@{|->}_{\iota_S}[d] \ar@{|->}^{id}[r]  & a  \\
G(S) \ar@{.>}_{\widehat{\iota_S}}[ur]  & &  &  [a] \ar@{|.>}_{\widehat{\iota_S}}[ur]  &  } $$ 
As $\widehat{\iota_S} \circ \iota_S = id$, $\iota_S$ is 1-1.  
Since $G(S)$ is generated by $\iota_S(S)$ and $S$ is a group, $\iota_S$ is onto. 
\item 
$S$ is a semigroup, $S^+=S \cup \{ *\}$ is a monoid  where $* \cdot s = s = s \cdot *$ and $* \cdot * = *$\\
Define  $\alpha: S \rightarrow S^+$ as $ s \mapsto s $, and define a monoid homomorphism $\beta : S^+ \rightarrow G(S) $ as 
$\beta : s \mapsto [s]$ and $\beta : * \mapsto 0$. Then the diagram below commutes 
$$ \xymatrix{ 
 S^+ \ar@{->}[r]^{\beta} \ar@{->}[d]  & G(S)  \ar@{->}[d] \\
   G(S^+) \ar@{->}[r]^{\widehat{\iota_{\beta}}} & G(S) }
$$
Consider the following diagrams where there exists a unique $\overset{\sim}{\iota_S}$ by the universality of $S^+$ and there exists a unique $\widehat{\overset{\sim}{\iota_S}}$ by the universality of $G$, 
$$ \xymatrix{ 
 S \ar@{->}[r]^{\iota_S}   \ar@{->}[d]_{\alpha_S}   & G(S)    & & S \ar@{->}[r]^{\iota_S}   \ar@{->}[d]_{\alpha_S} & G(S) 
 \\ S^+ \ar@{.>}[ur]_{\overset{\sim}{\iota_S}}  &     & &    
S^+ \ar@{.>}[ur]_{\overset{\sim}{\iota_S}} \ar@{->}[r]_{\iota_{S^+}} & G(S^+)  \ar@{.>}[u]_{\widehat{\overset{\sim}{ \iota_S}}}  }
$$
Now, in the following diagram,  
$$  \xymatrix{ S \ar@{^{(}->}[r]^{\alpha} 
\ar@{->}[d]^{\iota_S} & S^+  \ar@{->}[d]^{\iota_{S^+}} & \\
G(S) \ar@{->}[r]^{G(\alpha)}  & G(S^+) \ar@{->}[r]^{\overset{\sim}{\widehat{\iota_S}}} & G(S) }$$ 
we have $[s]^+ \mapsto [s]$ and $[s] \mapsto [s]^+$.
Both $G(S)$ and $G(S^+)$ are generated by $S$, $G(S) = \{ s+K_S=[s] \}$ and $G(S^+)=\{ s+K_{S^+}=[s]^+ \} $. Further, 
$* = * + * - ** \in K_{S^+}$, so $*+K_{S^+}=[*]^+=0$. Hence, 
$ G(\alpha_S) \circ \widehat{{\iota_S}} = id_{G(S^+)}$ and $ \widehat{{\iota_S}} \circ G(\alpha_S)= id_{G(S)}$.
Then $G(S) \cong G(S^+)$.
%\item Immediate from part (i) and (ii).
\item Let $S = M - \{1\}$, then $M=S^+$. The result follows from the parts above.
              \end{enumerate}
              \end{proof}
Some (counter)examples:
\begin{example}\label{ExG(S)}
\begin{enumerate}
\item $\iota:S \rightarrow G(S)$ is not necessarily one-to-one. For instance: \\
Let $S= <a: na=ma>$ where $n,m$ are distinct positive integers with $n>m>1$. 
Then $G(S)= \mathbb{Z} / (n-m)\mathbb{Z}$ and $|S|=n-1 > |G(S)|=n-m$.
Hence, $\iota_S$ is not 1-1.

\item \label{1-1} If $f: S \rightarrow T$ is 1-1, then $G(f): G(S) \rightarrow G(T)$ may not be 1-1 : 
Consider the additive semigroups $S=\mathbb{N} $  and $T= < v, \, w \, | \, v = v+w>$. 
($T$ is the semigroup of Example \ref{geoToeplitz}.) Let $f:S \rightarrow T$ be 
the semigroup monomorphism defined as $f(1) = w $, $f(n) = nw $. 
Then both $G(S)$ and $G(T)$ are $\mathbb{Z} $, but the group homomorphism $G(f)$ induced by $f$ between the Grothendieck groups 
$G(S)$ and $G(T)$ is the zero map which is not injective. However, if we take $S= \mathcal{U}(R)$ and $T=\mathcal{V}(R)$ for any ring $R$, then 
$G(f) : G(\mathcal{U}(R)) \rightarrow \mathcal{K}_0(R)$ is also a monomorphism which we will prove shortly in Proposition 
\ref{G(U)embedG(V)}. 

\item Let $a \in T$ be an element with infinite order, then it is possible that $\iota(a)$ has finite order: 
In the example above, $w  \in T$ is not torsion and $\iota(v) = \iota(v+w) = \iota(v) + \iota(w)$ in $G(T)$.
Hence, $\iota(w)=0$ and $\iota(w)$ is torsion in $G(T)=\mathbb{Z}$.
\end{enumerate}
\end{example}
We will proceed with an alternative construction of the Grothendieck group of a commutative semigroup which will also be used in the sequel. 
Let $S$ be a non-empty commutative semigroup, then 
%If $S = \emptyset$, then $G(S)$ is defined as the trivial group. Otherwise, 
the relation $\sim$ on $S \times S$, defined as $(a,b) \sim (c,d)$ if and only if there exists $x \in S$ such that $adx=bcx$, is an equivalence relation. The set of all equivalence classes of $S \times S$ is denoted by $E(S)$, and the operation $\cdot$ on $E(S)$ is
defined as $[(a,b)] \cdot [(c,d)] = [(ac,bd)]$. For any $\delta \in S$, $[(\delta,\delta)]$ is the identity element of $\cdot$, and 
$[(a,b)]$ is the inverse element for $[(b,a)]$. 
Hence, $E(S)$ is an abelian group where commutativity of $E(S)$ follows from the commutativity of $S$. 
Fix an $x \in S$, define $\phi_x: S \rightarrow E(S)$ which maps $a \in S$ to $[(ax, x)] \in E(S)$. 
Now, $\phi_x$ is a semigroup homomorphism. For any $x,y \in S$, $\phi_x = \phi_y$; hence $\phi$ is defined 
independent of the choice of $x \in S$. Also, consider the map $\hat{\phi}: G(S) \rightarrow E(S)$ induced by $\phi$ which maps 
$a+K_S \mapsto [(ax, x)]$.

Now, define $f: E(S) \rightarrow G(S)$ as $[(a, b)] \mapsto a-b+K_S$. So $f$ is  clearly a group homomorphism. It is straightforward to show that $f \circ \hat{\phi}$ and $\hat{\phi} \circ f$ are identity maps, i.e., 

 $[a,b] \mapsto a-b+K_S \mapsto [ax,x] [bx,x]^{-1}= [ax,x] [x,bx] = [ax^2,bx^2]= [a,b]$ and also,
  
$ a+K_S \mapsto [ax,x] \mapsto ax-x+K_S = a+x-x+K_S=a+K_S$.
So $E(S), G(S)$ are (naturally and canonically) isomorphic groups. We are going to identify $E(S)$ with $G(S)$ and use them interchangibly from now on. 

Note that $(\mathcal{V}^*(R))^+ = \mathcal{V}(R)$ and
recall that $\mathcal{K}_0(R)$ is defined as the Grothendieck group of the monoid $\mathcal{V}(R)$. 

Compare the following proposition with Example \ref{ExG(S)} part (\ref{1-1}).
\begin{proposition}\label{G(U)embedG(V)}
For any ring $R$, consider $\mathcal{U}(R)$ and $\mathcal{V}(R)$ as defined above, and $f$ the inclusion map, then the induced map
$G(f) : G(\mathcal{U}(R)) \rightarrow \mathcal{K}_0(R)$ is injective. 
%ie. $\hat{f} : G(S) \rightarrow G(T)$ is also a monomorphism.
\end{proposition}
\begin{proof}
Let us take an element $m[R]_{G(\mathcal{U}(R))}$ for some $n \in \Z$, from the kernel of $G(f)$.
WLOG we may choose $n \in N$ as $G(\mathcal{U}(R))$ is a group. Then
$ 0= G(f)(m[R]_{G(\mathcal{U}(R))})  = m[R]_{G(T)}$. 
We have the following commutative diagram  
$$ \xymatrix{ 
 S \ar@{^{(}->}[r]^f \ar@{->}[d]  & \mathcal{V}(R)  \ar@{->}[d]   & m[R]\neq 0 \ar@{|->}[r]^f \ar@{|->}[d]  & m[R]  \ar@{|->}[d] \\
 G(\mathcal{U}(R)) \ar@{->}[r]^{G(f)}   &  \mathcal{K}_0(R)  &  m[R]_{G(\mathcal{U}(R))} \ar@{|->}[r]^{G(f)}   &  [(m[R],0)] = [(0,0)]}
$$
We want to show that $m[R]_{G(\mathcal{U}(R))} =0 $.

$m[R] + 0 + [X]= 0+ [X] $  in $ \mathcal{V}(R)$, that is, $R^m \oplus X \cong X$ for some finitely 
generated projective module over $R$.
Then there exists an  $R$-module $Y$ such that $X \oplus Y = R^n$ for some $n \in \N$.
$R^m \oplus X \oplus Y \cong X \oplus Y$, i.e., $R^m \oplus R^n \cong R^n$, so 
$(n+m) [R] = n[R]$ in $\mathcal{V}(R)$, hence in $\mathcal{U}(R)$. Chasing the diagram we get 
$(n+m) [R]_{G(\mathcal{U}(R))} = n[R]_{G(\mathcal{U}(R))}$ in $G(\mathcal{U}(R))$. Since $G(\mathcal{U}(R))$ 
is a group we have inverses and get
$m[R]_{G(\mathcal{U}(R))} =0 $. Hence, $G(f)$ is a monomorphism. 

\end{proof}
For the isomorphism class of the free module of rank one $[R] \in \mathcal{V}(R)$, we use the following notation. 
Let $\iota([R]):=[R]_{\Z} \in \mathcal{K}_0(R)$ and $[R]_{\Q}:= [R]_{\Z} \otimes_\Z 1 \in \mathcal{K}_0(R) \otimes_\Z \Q$.
If an element $[P]$ is torsion in $\mathcal{V}^*(R)$, then $\iota([P])$ has finite order in 
$\mathcal{K}_0(R)$.

\medskip 

\subsubsection{Graph Semigroups} 
$\Gamma = (E, V, s, t)$ will denote a {\bf digraph} with vertex set $V$, edge set $E$, source and target functions 
$s$, $t: E \rightarrow V$. 
A {\bf sink} is a vertex $v$ for which the outgoing edge set $s^{-1}(v)$ 
is empty and a {\bf source} is a vertex $v$ for which the incoming edge set $t^{-1}(v)$ is empty. In particular, an isolated vertex is both a source and a sink. Throughout this note we will assume $\Gamma$ is a finite digraph, that is both $V$ and $E$ are finite. 

A {\bf separated graph} $(\Gamma,\Pi)$ is a digraph $\Gamma$ together with a partition $\Pi$ of $E$ 
which is finer than $\{s^{-1}(v) | v \in V, \, v \text{ not a sink} \}$. 

For a given triple $(\Gamma,\Pi,\Lambda)$ where $\Lambda \subseteq \Pi$, $S(\Gamma,\Pi,\Lambda)$ is defined to be the commutative semigroup generated by $\Omega:= V \sqcup ( \Pi \backslash \Lambda)$ subject to the relations $\{ R_X\}_{X \in \Pi}$ such that
$$sX = \sum_{ e\in X } t(e) \quad \mbox{ for every } X \in \Lambda, \mbox{ and}$$
$$sX = X + \sum_{ f\in X } t(f) \quad \mbox{ for every } X \in \Pi \backslash \Lambda.$$ 
The Grothendieck group of $\N\Omega  \setminus \{{ \mathbf 0} \}$ is $\Z\Omega$.
The Grothendieck group of $S(\Gamma,\Pi, \Lambda)$ is $\Z\Omega / f(R_X)$  where $X \in \Pi$ which we will denote as $S_{\Z}(\Gamma,\Pi, \Lambda)$. Notice also that $f(R_X)$ is exactly the $\Z$-span of the relations $R_X$. 
\begin{example}\label{L(2,5)semigroup}
Consider the following separated graph $\Gamma$ with vertex set $V=\{v,w\}$ and edge set $E=\{e_1,e_2,f_1,f_2,f_3,f_4,f_5\}$ with $s(e)=v$ 
and $t(e)=w$ for all $e \in E$, $\Pi = \{ X,Y\}$ where $X=\{e_1,e_2\}$, $Y=\{f_1,f_2,f_3,f_4,f_5\}$ and $\Lambda = \Pi$. We marked the edges in $X$ with red, and the edges in $Y$ with blue in the graph. The set $\Omega = V $ and $v=2w$ and $v=5w$ are the relations on 
$\N\Omega  \setminus \{{ \mathbf 0} \}$.
 
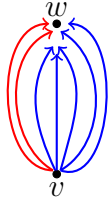
\begin{figure}[h]
\begin{tikzpicture}
%\draw[->, blue, thick] (-0.9,0.05) -- (0.9,1);
%\draw[->, red, thick] (-0.9,-0.05) -- (0.9,-1);
%\draw [fill=black] (-1,0) circle [radius=0.05];
\draw [fill=black] (1,1) circle [radius=0.05];
\draw [fill=black] (1,-1) circle [radius=0.05];
%\node at (-1.2,0) {$u$};
\node at (1,1.2) {$w$};
\node at (1,-1.2) {$v$};
\draw[->, blue, thick] (1,-0.95) to [out=90,in=270] (1,0.6);
\draw[->, blue, thick] (1.05,-0.95) to [out=75,in=310] (1.02,0.7);
\draw[->, blue, thick] (1.05,-0.95) to [out=30,in=330] (1.05,0.9);
\draw[->, blue, thick] (0.98,-0.95) to [out=115,in=235] (0.95,0.7);
\draw[->, blue, thick] (1.05,-0.95) to [out=340,in=0] (1.15,1);
\draw[->, red, thick] (0.95,-0.95) to [out=150,in=200] (0.95,0.9);
\draw[->, red, thick] (0.95,-0.95) to [out=180,in=180] (0.90,1);
\end{tikzpicture}
\end{figure}
\end{example}
The graph semigroup $S(\Gamma, \Pi)=S(\Gamma, \Pi,\Pi)$ is the semigroup of Example \ref{geoL(2,5)}.

\begin{example}\label{Toeplitz}
Consider the following digraph $\Gamma$: 
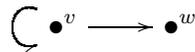
$$ \xymatrix{\bullet^{v} \ar@(ul,dl) \ar[r] &  \bullet^w  }$$ where $\Pi = \{E\} $ and $\Lambda=\Pi$. 
Now $\Omega = V $ and $v=v + w$ is the only relation. %on $\N\Omega  \setminus \{{ \mathbf 0} \}$. 
The graph semigroup is the semigroup of Example \ref{geoToeplitz}.

\end{example}
\section{Separated Cohn-Leavitt Path Algebras and IBN}
We can rephrase IBN property for a ring $R$ in terms of $\mathcal{V}^*(R)$ as follows: 
$R$ has IBN if and only if for every pair of distinct positive integers $m\neq m'$ we have $m[R] \neq m'[R]$ as elements of $\mathcal{V}^*(R)$. Hence, $[R]$ has infinite order in $\mathcal{V}^*(R)$ if and only if $R$ is IBN. Moreover, $[R]$ has infinite order in $\mathcal{U}(R)$. 
For any ring $R$, the property of having IBN is detected by $\mathcal{K}_0(R)$.
\begin{proposition} \label{RIBNU(R)inf}\cite[Corollary 5.7]{AGP}For any ring $R$, the followings are equivalent:

(1) $R$ has IBN; 

(2) $\mathcal{U}(R)$ is infinite; 
 
(3) $[R]_{\Q}$ is nonzero in $\mathcal{K}_0(R)\otimes \mathbb{Q}$.
\end{proposition}
\begin{proof}
Now, $R$ has IBN if and only if $\mathcal{U}(R)$ is infinite is clear from the definitions. By Proposition \ref{G(U)embedG(V)},  
$G(\mathcal{U}(R)) \rightarrow \mathcal{K}_0(R)$ induced from the semigroup inclusion 
$\mathcal{U}(R) \rightarrow \mathcal{V}(R)$ is injective. Now, if $\mathcal{U}(R)$ is an infinite cyclic semigroup then its Grothendieck group 
$G(\mathcal{U}(R))$ is an infinite cyclic group generated by $[R]$ by Fact \ref{PropertiesG(S)} parts (\ref{ifScyclic,G(S)cyclic}) and (\ref{G(S)inf-iffSinf}). So
$[R]_{\mathbb{Z}}$ in $\mathcal{K}_0(R)$ also has infinite order. Hence $[R]_{\mathbb{Q}}$ in $\mathcal{K}_0(R)\otimes \mathbb{Q}$ is nonzero. 

If $\mathcal{U}(R)$ is a finite cyclic semigroup of type $(m,n)$, then $[R]_{\mathbb{Z}}$ in $G(\mathcal{U}(R))$ generates a finite cyclic group of order $n-m$, so $[R]_{\mathbb{Z}}$ in 
$\mathcal{K}_0(R)$ also has order $n-m (> 0)$ and its image in $\mathcal{K}_0(R)\otimes \mathbb{Q}$ is 0. 
\end{proof}
\begin{definition}\label{definition}  
{\rm For a given triple $(\Gamma,\Pi,\Lambda)$, the {\bf separated Cohn-Leavitt path algebra} over the field $K$, denoted by  
$CL_K(\Gamma,\Pi,\Lambda)$, is the algebra which is generated by the sets
$\{v\mid v\in V\}$, $\{e,e^*\mid e\in E\}$, which satisfy the following relations:

(V)  $vw = \delta_{v,w}v$ for all $v,w\in V$, \  %(i.e., $\{v\mid v\in E^0\}$ is a set of orthogonal idempotents),

(E1) $s(e)e=et(e)=e$ for all $e\in E$,

(E2)  $t(e)e^*=e^*s(e)=e^*$ for all $e\in E$, and 

(CK1)  For all $X \in \Pi$, $e^*f=\delta _{ef}t(e)$ for every $e,f\in X$.

(CK2)  For all $X \in \Lambda$, $sX=\sum _{\{ e\in X \}}ee^*$ for every $e \in X$. \\ 

For ease of notation we will drop $K$, and use $CL(\Gamma,\Pi,\Lambda)$. }
\end{definition}

If $\Lambda = \Pi$, then $L(\Gamma,\Pi):=CL(\Gamma,\Pi,\Pi)$ is the {\bf separated Leavitt path algebra} over the separated graph 
$(\Gamma,\Pi)$; 

if $\Lambda = \emptyset$, then $C(\Gamma,\Pi):=CL(\Gamma,\Pi,\emptyset)$ is the {\bf separated Cohn path algebra} over the separated graph 
$(\Gamma,\Pi)$;

if $\Lambda = \Pi = \{s^{-1}(v) | v \in V, \, v \text{ not a sink} \}$, then $L(\Gamma):=CL(\Gamma,\Pi,\Pi)$ is the {\bf Leavitt path algebra} over the graph $\Gamma$, and 

if $\Lambda = \emptyset$ and $\Pi =\{s^{-1}(v) | v \in V, \, v \text{ not a sink} \}$, then  $C(\Gamma):=CL(\Gamma,\Pi,\emptyset)$ is the {\bf Cohn path algebra} over the graph $\Gamma$. 

\medskip
 
The connection between the graph semigroup and the non-stable $\mathcal{K}$-theory of the separated Cohn-Leavitt path algebra 
$CL(\Gamma, \Pi, \Lambda)$ is provided by the following lemma.
\begin{lemma}\label{iso}
For a given triple $(\Gamma, \Pi,\Lambda)$, let $L$ denote $CL(\Gamma, \Pi, \Lambda)$. Then 
$$(sX)L \cong \bigoplus_{e \in X} t(e)L  \qquad \mbox{ for }  X \in \Lambda,$$ 
$$(sY)L \cong \bigoplus_{e \in Y} t(e)L \bigoplus (sY - \sum_{e \in Y} ee^*)L \qquad \mbox{ for }  Y \in \Pi \backslash \Lambda.$$ 
\end{lemma}
\begin{proof}
Notice that $sY - \sum_{e \in Y} ee^*$ is an idempotent, as 
$$(sY - \sum_{e \in Y} ee^*)(sY - \sum_{e \in Y} ee^*)= 
sY - sY\sum_{e \in Y} ee^* - (\sum_{e \in Y} ee^*)sY + (\sum_{e \in Y} ee^*)^2 $$
$$ sY - \sum_{e \in Y} ee^* - (\sum_{e \in Y} ee^*) + (\sum_{e \in Y} ee^*) = sY - \sum_{e \in Y} ee^* $$ by using the (CK2) relation. 
Hence, $(sY - \sum_{e \in Y} ee^*)L$ is a cyclic projective $L$-module, so is $vL$ for every vertex $v$ of $\Gamma$.

For $X \in \Lambda$, define the following maps $f_1 : (sX)L \rightarrow \bigoplus_{e \in X} t(e)L $  as 
$m \mapsto (e^*m)_{e \in X} $ and $f_2 :  \bigoplus_{e \in X} t(e)L \rightarrow (sX)L $ as 
$ (m_e)_{e \in X} \mapsto \sum_{e \in X}em_e$ where $f_1 \circ f_2$ and $f_2 \circ f_1$ are both the identity functions. 
Similarly, define 
$g_1 : (sY)L \rightarrow (\bigoplus_{e \in Y} t(e)L) \bigoplus (sY-\sum_{e \in Y} ee^*)L$  as 
$m \mapsto ( (e^*m)_{e \in Y}, (sY- \sum_{e \in Y} ee^*)m ) $ and 
$g_2 :  (\bigoplus_{e \in X} t(e)L) \bigoplus (sY-\sum_{e \in Y} ee^*)L \rightarrow (sY)L $ as 
$ ((m_e)_{e \in Y},m_Y) \mapsto \sum_{e \in Y}em_e + m_Y$ where 
$g_1 \circ g_2$ and $g_2 \circ g_1$ are also the identity functions. 
\end{proof}

For a given triple $(\Gamma, \Pi,\Lambda)$, let $L$ denote $CL(\Gamma, \Pi, \Lambda)$. Then 
there exists a semigroup homomorphism $\alpha : S(\Gamma, \Pi, \Lambda) \rightarrow \mathcal{V}(L)$ that maps 
$v \mapsto [vL]$ and $Y \mapsto [(sY - \sum_{e \in Y} ee^*)L]$ for $v \in V, $ and $Y \in \Pi \backslash \Lambda$.
The map $\alpha$ defined on the generators of $ \N\Omega $ can be extended additively. 
We can show that $\alpha$ preserves the relations on $S(\Gamma, \Pi, \Lambda)$ by Lemma \ref{iso}.
Hence, $\alpha$ is a semigroup homomorphism. 
In fact, this homomorphism $\alpha$ is an isomorphism as proved in \cite[Theorem 4.3]{AG} which we will use, but not prove here.

\medskip 

Now, we are ready to prove our main result. 
\begin{theorem}\label{MainThmforseparatedLPA} For a given triple $(\Gamma, \Pi,\Lambda)$, let $L$ denote the Cohn-Leavitt path algebra, 
$CL(\Gamma, \Pi, \Lambda)$ over the triple. Then   
$L$ is IBN if and only if $ \sum_{v \in V}v$  is not in the $\Q$-span of the relations 
$\{ sX - \sum_{e \in X}te\}_{ X \in \Lambda}$  in $\Q\Omega$.
\end{theorem}
\begin{proof}
In the following diagram of commutative additive semigroups,  
$$ \xymatrix{ \N\Omega  \setminus \{{ \mathbf 0} \} \ar@{^{(}->}[r] \ar@{->}[d]^f  & \Z\Omega \ar@{^{(}->}[r] \ar@{->}[d]^g &  
\Q\Omega \ar@{->}[d]^h   \\
 S(\Gamma,\Pi,\Lambda) \ar@{->}[r]^\alpha \ar@{->}[d]^{\cong}  & S_{\Z}(\Gamma,\Pi,\Lambda) \ar@{->}[r]^\beta 
\ar@{->}[d]^{\cong}&  S_{\Q}(\Gamma,\Pi,\Lambda)  \ar@{->}[d]^{\cong} \\ 
\mathcal{V}^*(L) \ar@{->}[r] & \mathcal{K}_0(L) \ar@{->}[r] & \mathcal{K}_0(L)\otimes \Q  }
$$
$f,g,h$ are epimorphisms with $Ker(g) = \Z-span \{ R_X \}_{X \in \Pi}$ and $Ker(h) = \Q-span \{ R_X\}_{X \in \Pi}$ . 
Now, $S(\Gamma,\Pi,\Lambda) \cong  \mathcal{V}^*(L)$ by \cite[Theorem 4.3]{AG} and since taking Grothendieck group and tensoring with $\Q$ are both functorial $S_{\Z}(\Gamma,\Pi,\Lambda) \cong  \mathcal{K}_0(L)$
and $S_{\Q}(\Gamma,\Pi,\Lambda) \cong \mathcal{K}_0(L)\otimes \Q $.

$L$ has IBN if and only if $\mathcal{U}(L)$ is infinite if and only if $[L]$ has infinite order in $\mathcal{U}(L)$ 
by Proposition \ref{RIBNU(R)inf}. 
Moreover, $\mathcal{U}(L)$ is cyclic, so $\mathcal{U}(L) \cong \mathbb{N}$ and $G(\mathcal{U}(L)) \cong \mathbb{Z}$. 
By Proposition \ref{G(U)embedG(V)}, as $G(\mathcal{U}(L)) $ is embedded in $\mathcal{K}_0(L)$; 
$[L]_{\Z}$ also has infinite order in $\mathcal{K}_0(L)$. This is identical to: 
$[L]_{\Q}$ is non-zero in $\mathcal{K}_0(L)\otimes \Q $. 
On the other hand, by \cite[Theorem 4.3]{AG}, $S(\Gamma, \Pi, \Lambda) \cong \mathcal{V}(L)$. 
$ [\sum_{v \in V}v] $ is non-zero in $G(S(\Gamma,\Pi,\Lambda)) \otimes \Q $. Equivalently, 
$ \sum_{v \in V}v $ is not in the $\Q$-span of the relations $\{R_X\}_{X \in \Pi}$.

If $ \sum_{v \in V}v$ is in $\Q$-span of $\{R_X\}_{X \in \Pi}$, then $ \sum_{v \in V}v = \sum_{X \in \Pi}\alpha_X R_X $ for some coefficients 
$\alpha_X \in \Q$. However, we are going to show that the coefficient of $R_Y$, for any $Y \in \Pi \backslash \Lambda$ has to be zero. 

For any fixed $Y \in \Pi\backslash \Lambda$, consider the projection map 
$pr_Y : \Q\Omega  \rightarrow \Q$. Note that $\Q\Omega = \Q V \oplus \Q(\Pi\backslash \Lambda)$, so $pr_Y(\sum_{v \in V}v) =0$, hence 
$pr_Y(\sum_{X \in \Pi}\alpha_X R_X) =0$. Since $pr_Y(R_X) = \delta_{XY}$, then in this sum $X \neq Y$. 
Therefore,  $X \not \in \Pi \backslash \Lambda$, i.e., $X \in \Lambda$. 
\end{proof}
The original proof of Theorem \ref{MainThmforseparatedLPA} was more geometric. It was presented by the first named author at the CIMPA Research School Izmir, 
Turkey, in July 2015. (The videos of these talks have been publicly available as instructed by CIMPA.) We will sketch this geometric proof below. 

\begin{2ndproof} 
If $\sum_{v\in V} v$ is in the $\Q$-span of the relations of the semigroup of $(\Gamma, \Pi, \Lambda)$, then there exists 
$k \in \Z$, so that $k\sum_{v\in V} v$ is in the $\Z$-span of the relations $R_X$, for $X \in \Pi$.
Hence, there is a path from $\sum_{v\in V} v$ to $(k+1)(\sum_{v\in V} v)$ in the $\Z\Omega$ representation of 
$S(\Gamma, \Pi, \Lambda)$. 
Now, if this path is totally contained in $\N\Omega \setminus \{{\bf 0}\}$, then $L^1 \cong L^{k+1}$ where $L$ is the separated Cohn-Leavitt path algebra over $(\Gamma, \Pi, \Lambda)$. If not, as the geometric representation of $S(\Gamma, \Pi, \Lambda)$ is translation invariant, we can move the path from $\Z\Omega$ to $\N\Omega \setminus \{{\bf 0}\} $ translating by 
$t\sum_{v\in V} v$ by sufficiently large $t$. Hence, we obtain a path from 
$(1+t)(\sum_{v\in V} v)$ to $(k+t)(\sum_{v\in V} v)$ in the $\N\Omega \setminus \{{\bf 0}\}$ representation of $S(\Gamma, \Pi, \Lambda)$. 
Again, $L^{1+t} \cong L^{k+t}$. Thus, $CL(\Gamma, \Pi, \Lambda)$ is not IBN.

Conversely, if $L=CL(\Gamma, \Pi, \Lambda)$ is not IBN, then there exist distinct $k,l \in \N$ with $L^k \cong L^l$. Then 
there is a path from $k(\sum_{v\in V} v)$ to $l(\sum_{v\in V} v)$ in the $\N\Omega \setminus \{{\bf 0}\}$ representation of 
$S(\Gamma, \Pi, \Lambda)$. (or both vectors are in the same path component)  
Then $(k-l)\sum_{v\in V} v$ is in the $\Z$-span of the relations of the semigroup of $(\Gamma, \Pi, \Lambda)$, hence 
$\sum_{v\in V} v$ is in the $\Q$-span of the relations $R_X$ for $X \in \Pi$. In fact, $\sum_{v\in V} v$ is in the $\Q$-span of the relations $R_X$ for $X \in \Lambda$ can be shown as in the last paragraph of the first proof.  
\end{2ndproof}
   
Now, an immediate corollary to the main result is that any separated Cohn path algebra has IBN. 
This generalizes the fact that any Cohn path algebra has IBN which was proven in \cite[Theorem 9]{AK}.
\begin{corollary}\label{sepCohnIBN} 
Any separated Cohn path algebra has IBN. 
\end{corollary}

\section{Corner Subalgebras of Separated Cohn-Leavitt Path Algebras and IBN} 
The next theorem shows that the non-stable $\mathcal{K}$-theory of a corner ring is given explicitly in terms of the non-stable $\mathcal{K}$-theory of the ring. 
\begin{theorem}\label{V(xRx)1-1V(R)} Assume $R$ is a ring and $\alpha \in R$ is an idempotent.
Then the semigroup homomorphism induced by $ - \otimes_{\alpha R\alpha} \alpha R $ from $\mathcal{V}(\alpha R\alpha)$ into 
$\mathcal{V}(R)$ is 1-1. Moreover, 
$$\mathcal{V}(\alpha R\alpha) %= \overline{\mathcal{U}(\alpha R\alpha)}  
\cong \overline{<[\alpha R]>}   \subseteq \mathcal{V}(R).$$
\end{theorem} 
\begin{proof}
Assume $R$ is a ring (not necessarily unital) and $\alpha\in R$ is an idempotent. 
$vR$ is a projective right $R$-module. Let $\alpha R\alpha$ be the associated corner subring of $R$. Note that 
$\alpha$ is the unit element in $\alpha R\alpha$, so 
$\alpha R\alpha$ is a unital ring. However, $\alpha R\alpha$ is not a unital subring of $R$, as $\alpha$ is not the unit of $R$ unless 
$\alpha R\alpha=R$. $\alpha R$ is finitely generated $(\alpha R\alpha,R)$-bimodule and as a right $R$-module it is finitely generated 
(in fact, cyclic) projective. 

Our aim is to prove that $\mathcal{V}(\alpha R\alpha)$ is embedded into $\mathcal{V}(R)$. 
Consider $\mathcal{M}_{\alpha R\alpha}$ as the category of right $\alpha R\alpha$-modules, $\mathcal{M}_R$ as the category of right 
$R$-modules and $F$ as the functor $-\otimes_{\alpha R\alpha} \alpha R$ from $\mathcal{M}_{\alpha R\alpha }$ to $\mathcal{M}_R$. Now, the functor $F$ induces a semigroup homomorphism $f$ 
from $\mathcal{V}(\alpha R\alpha)$ to $\mathcal{V}(R)$ as for any isomorphism class of finitely generated projective $R$-module $P$, $$ [P] \overset{f}{\mapsto} [P \otimes_{\alpha R\alpha } \alpha R] .$$ Also, define $G$ as the functor $Hom_R(\alpha R,-)$ from 
$\mathcal{M}_R$ to $\mathcal{M}_{\alpha R\alpha}$. Notice that $\alpha R$ is mapped to $Hom_R(\alpha R,\alpha R)$ which is isomophic to 
$\alpha R\alpha$ as a $\alpha R\alpha$-module. 
We can also show that the functor $H :  \mathcal{M}_R  \rightarrow \mathcal{M}_{\alpha R\alpha}$ defined as $M \mapsto M\alpha$ for any $R$-module $M$, is equivalent to the functor $G$. Hence, replacing $G$ with $H$, for any projective $R$-module $P$ we get   
$$ P \overset{F}{\mapsto} P \otimes_{\alpha R\alpha } \alpha R \overset{H}{\mapsto} (P \otimes \alpha R)\alpha .$$
Take $\alpha R\alpha \hookrightarrow \alpha R$ inclusion map, and tensor with a projective module $P$, that is:
$$ P \otimes \alpha R\alpha \overset{\beta}{\rightarrow} P \otimes \alpha R.$$ As $P$ is projective, it is flat; so $\beta$ is 1-1. 
Moreover, $P \cong P \otimes \alpha R\alpha \cong \beta( P \otimes \alpha R\alpha ) = (P \otimes \alpha R)\alpha$.   
We proved that $(P \otimes \alpha R)\alpha \cong P$, and conclude that their isomophism classes in $\mathcal{V}(R)$ are the same. 
The semigroup homomorphism $h$, induced by the functor $H$ gives
$$ [P] \overset{f}{\mapsto} [P \otimes_{\alpha R\alpha} \alpha R] \overset{h}{\mapsto} [P].$$ Thus, the semigroup homomorphism $f$ is 1-1.  

In particular,$$ <[\alpha R\alpha]> = \mathcal{U}(\alpha R\alpha) \cong  <[\alpha R]> \subseteq \mathcal{V}(R).$$ 

For the last statement, we will show that the closure of $<[\alpha R]>$ in $\mathcal{V}(R)$ is isomorphic to the closure of 
$\mathcal{U}(\alpha R\alpha)$ in $\mathcal{V}(\alpha R\alpha)$.

Consider a (right) $R$-module $P$ where $P \oplus Q \cong (\alpha R)^k$ for some positive integer $k$ and $R$-module $Q$, i.e., $[P]$ is an element of $\overline{<[\alpha R]>}$ in $\mathcal{V}(R)$.  
We can identify $P$ with its isomorphic copy as a submodule in $(\alpha R)^k$.

Consider the maps $\mu: Hom_R(\alpha R, P) \otimes \alpha R \rightarrow P$, and $\eta: P \rightarrow Hom_R(\alpha R, P) \otimes \alpha R $ 
mapping $\mu: f \otimes x \mapsto fx$ and $\eta: y \mapsto \sum_{j=1}^k \phi_j \otimes \pi_j y$ where $\pi_j: P \rightarrow \alpha R$ is the (restriction of the) projection map to the $j^{th}$ component and $\phi_j: \alpha R \rightarrow P $ defined as the composition of the inclusion of $\alpha R$ into $(\alpha R)^k$ as the $j^{th}$ coordinate with the projection from $(\alpha R)^k$ to $P$.
%$$x \mapsto \begin{cases}  (0,\cdots, \overset{j^{th}}x, 0, \cdots,0) & \mbox{ if } x \in \pi_j(P)\\ 0 & 
%\mbox{ if } x \not \in \pi_j(P) \end{cases}$$

Note that $\pi_j \circ f \in End_R(\alpha R) \cong \alpha R \alpha $ and 
 $Hom_R(\alpha R, P)$ is a right $End_R(\alpha R)$-module. (All tensor products are over $\alpha R \alpha$.) Further, 
$(\sum_{j=1}^k \phi_j \pi_j) $ is the identity map on $P$. We get 
$$\eta \mu (f \otimes x ) =   \sum_{j=1}^k \phi_j \otimes \pi_j fx = (\sum_{j=1}^k \phi_j \pi_j) f \otimes x = f \otimes x  \mbox{ and} $$ 
$$ \mu \eta (y) =   \sum_{j=1}^k \phi_j \otimes \pi_j y = \sum_{j=1}^k \phi_j \pi_j y = y.$$ 
%The pre-image of $P$ under the map $- \otimes \alpha R$ is $Hom_R(\alpha R, P)$ which is a right $\alpha R\alpha$-module. 
We need to show that $Hom_R(\alpha R, P)$ is finitely generated projective. As $P \oplus Q \cong (\alpha R)^k$ for some positive integer $k$ and $R$-module $Q$, applying the $Hom_R(\alpha R, -)$ functor we get
$$Hom_R(\alpha R, P) \oplus Hom_R(\alpha R, Q) \cong Hom_R(\alpha R, (\alpha R)^k) = (Hom_R(\alpha R, \alpha R) )^k\cong (\alpha R\alpha)^k .$$
%So $Hom_R(\alpha R, P)$ is finitely generated projective. 
%As $Hom_R(\alpha R, P)$ is a quotient of a finitely generated module, it is finitely generated. 
Hence, $Hom_R(\alpha R, -)$ induces an isomorphism from $\overline{<[\alpha R]>} $ to $\mathcal{V}(\alpha R\alpha)$.  
\end{proof}

\medskip

\begin{corollary}\label{vRv-vR} 
Assume $R$ is a ring and $\alpha \in R$ is an idempotent. Then for $m < n$,\\
$(\alpha R\alpha)^m \cong (\alpha R\alpha)^n$ as $\alpha R \alpha$-modules if and only if 
$(\alpha R )^m \cong (\alpha R)^n$ as $R$-modules. Hence, $\alpha R\alpha$ has IBN if and only if $[\alpha R]$ has infinite order in $\mathcal{V}(R)$. 
  Otherwise, $\alpha R\alpha$ is non-IBN of type $(m,n)$ if and only if $[\alpha R]$ is torsion of type $(m,n)$ in $\mathcal{V}(R)$. 
\end{corollary}

\begin{remark}\label{vR-vRv} Note that Corollary \ref{vRv-vR} can be  proven directly as:
%Clearly, $End_R(\alpha R) \cong  \alpha R\alpha$ as rings.
%Hence,
$$Hom_R((\alpha R)^m,(\alpha R)^n) \cong M_{n \times m}(\alpha R\alpha ) 
\cong Hom_{\alpha R\alpha}((\alpha R\alpha)^m,(\alpha R\alpha)^n).$$  However, this direct proof does not give 
$\mathcal{V}(\alpha R\alpha)$ as a submonoid of $\mathcal{V}(R)$.
\end{remark}

\begin{theorem}\label{MainThmforCornerLPA} Assume $L$ is the Cohn-Leavitt path algebra of $(\Gamma, \Pi, \Lambda)$ and $\alpha \in L$ is an idempotent of the form $\alpha = \sum_{v \in H}v$ where $H$ is a subset of the vertex set $V$. If $\alpha$ is not in the span of the relations $\{ sX - \sum_{e \in X}te\}_{ X \in \Lambda}$ in $\Q \Omega$, 
then $\alpha R\alpha$ has IBN.
\end{theorem}

\begin{proof} Let $R_X:= sX - \sum_{e \in X}te$. 
If $ \alpha = \sum_{v \in H}v$ is in $\Q$-span of $\{R_X\}_{X \in \Pi}$, then $ \alpha = \sum_{v \in H}v$ is in $\Q$-span of 
$\{R_X\}_{X \in \Lambda}$ as proven in the proof of Theorem \ref{MainThmforseparatedLPA}. 
So we can assume $ \alpha= \sum_{v \in H}v $ is not in the $\Q$-span of the relations $\{R_X\}_{X \in \Pi}$. Equivalently, 
$ [\sum_{v \in H}v]_\Q $ is non-zero in $G(S(\Gamma,\Pi,\Lambda)) \otimes \Q $. On the other hand, by \cite[Theorem 4.3]{AG}, $S(\Gamma, \Pi, \Lambda) \cong \mathcal{V}(L)$, and $[\oplus_{v \in H} vL]_{\Q}$ is non-zero in $\mathcal{K}_0(L)\otimes \Q $. Equivalently, 
$[ \oplus_{v \in H} vL]_{\Z}$ also has infinite order 
in $\mathcal{K}_0(L)$ if and only if $[ \oplus_{v \in H} vL]$ has infinite order in $\mathcal{V}(L)$. Then, 
$[ \alpha L \alpha]$ has infinite order in $\mathcal{V}( \alpha L \alpha)$ if and only if 
$[ \alpha L \alpha]$ has infinite order in $\mathcal{U}( \alpha L \alpha)$ if and only if $\alpha L\alpha$ has IBN 
by Proposition \ref{RIBNU(R)inf}.

\end{proof}

\medskip 

\begin{example}\label{Toeplitz2}
Consider the Toeplitz algebra $L$ as the Leavitt path algebra of the following graph $\Gamma$ of Example \ref{Toeplitz} 
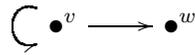
$$ \xymatrix{\bullet^{v} \ar@(ul,dl) \ar[r] &  \bullet^w 
 }$$
 
%$\mathcal{K}_0(L) \otimes \Q$ 
Consider the following diagram, 
$$  S(\Gamma)  \longrightarrow    S_{\Z}(\Gamma) \longrightarrow  S_{\Q}(\Gamma) $$ 
%$$[v] \mapsto [v]_{\Z} \mapsto [v]_{\Q} \neq 0 $$ 
$$[v] \longmapsto [v]_{\Z} \longmapsto [v]_{\Q} \neq 0 $$ 
$$ [w] \longmapsto  [w]_{\Z} =0 \longmapsto 0  $$ 
%$$[v] \longmapsto [v]_{\Z} \longmapsto [v]_{\Q} \neq 0 $$ 

Now, $[v] \in S(\Gamma)$ is not torsion, and neither is $[v]_{\Z}$ or $[v]_{\Q}$.  
%That is, $[wL]$ in $\mathcal{V}(L)$, is not torsion, but it is mapped to  $[wL]_{\Z}$ which is zero in $\mathcal{K}_0(L)$. 
Hence, $v$ is not in the $\Q$-span of the relation $R_1: v =v +w$, then by Theorem \ref{MainThmforCornerLPA}, $vLv$ has IBN.

However, $[w] \in S(\Gamma)$ is not torsion, so by Corollary \ref{vRv-vR}, $wLw$ has IBN. But 
$[w]_{\Z} = 0 \in S_{\Z}(\Gamma)$ is torsion. Hence, $w$ is in the $\Q$-span of the relation $R_1: v =v +w$, and Theorem \ref{MainThmforCornerLPA} is inconclusive. This example shows that the converse of Theorem \ref{MainThmforCornerLPA} is not true. 
\end{example}

In \cite{AK}, an example of two Morita equivalent rings, one having IBN and the other non-IBN is constructed.
%(atif Ex 12 sonu) Consider the graph of Example \ref{7} ). 
Next example provides Morita equivalent rings which are non-IBN, but are of different types. 
\begin{example} \label{E1} 
Consider the following separated graph $\Gamma$ with vertex set $V=\{v,w\}$ and edge set $E=\{e_1,\cdots, e_m,f_1,\cdots,f_n\}$ with $s(e)=v$ 
and $t(e)=w$ for all $e \in E$, $\Pi = \{ X,Y\}$ where $X=\{e_1,\cdots, e_m\}$, $Y=\{f_1,\cdots,f_n\}$ for $m<n$ positive integers and
$\Lambda = \Pi$. We want to figure out the nonstable $\mathcal{K}$-theory of $L:= L(\Gamma,\Pi)$, $vLv$ and $wLw$.  

\medskip 

The Leavitt algebra $L(m,n)$, the non-IBN algebra of type (m,n) constructed by Leavitt, is isomorphic to $wLw$ \cite{AG}. 
Since $\mathcal{V}(L)$ is generated by $[vL],[wL]$ subject to the relations $[vL]= m[wL] =n [wL]$, 
it follows that $[wL]$ generates $\mathcal{V}(L)$. Hence,  
$[wL]$ is a progenerator and $L$ is Morita equivalent to its corner algebra $wLw \cong End_R(wL)$. In fact, 
$L \cong M_{n+1}(wLw)$.

\medskip  

Also, $\overline{<[vL]>} = \mathcal{V}(L)$, since $[wL] \preceq [vL]$,we get that $vL$ is also a progenerator and $vLv$ is Morita equivalent to $L$. The element $[vLv]$ of $\mathcal{V}(vLv)$ maps to $[vL]$ in $\mathcal{V}(L)$. By Corollary \ref{vRv-vR}, $vLv$ is non-IBN of type 
$(m/d,n/d)$ where $d = gcd(m,n)$. By choosing $d \neq 1$, we get examples of non-IBN Morita equivalent rings of different types. 
%The set $\Omega = V $ and $v= mw$ and $v=n w$ are the relations on $\N\Omega  \setminus \{{ \mathbf 0} \}$.
\end{example}

\end{document}